\documentclass{amsart}

\theoremstyle{plain}
\newtheorem{theorem}{Theorem}
\newtheorem{corollary}[theorem]{Corollary}
\newtheorem{lemma}[theorem]{Lemma}

\theoremstyle{definition}
\newtheorem{example}[theorem]{Example}

\newtheorem{definition}[theorem]{Definition}

\newtheorem{remark}[theorem]{Remark}
\theoremstyle{remark}

\begin{document}
\title{Comaximal factorization lattices}

\author{Tiberiu Dumitrescu  and Mihai Epure}

\address{Facultatea de Matematica si Informatica,University of Bucharest,14 A\-ca\-de\-mi\-ei Str., Bucharest, RO 010014,Romania}
\email{tiberiu@fmi.unibuc.ro, tiberiu\_dumitrescu2003@yahoo.com }

\address{Simion Stoilow Institute of Mathematics of the Romanian AcademyResearch unit 5, P. O. Box 1-764, RO-014700 Bucharest, Romania}\email{mihai.epure@imar.ro, epuremihai@yahoo.com }

\maketitle
%
%
%

\section{Introduction}

In \cite{BH}, Brewer and Heinzer studied 
the (integral) domains $D$ having the property that each
proper ideal $A$ of $D$ has a comaximal ideal factorization $A = Q_1Q_2\cdots  Q_n$ (that is, $Q_1,...,Q_n$ are pairwise comaximal ideals)
and  the $Q_i$'s  have some additional property.  
They proved in \cite[Theorems 2, 8 and 9]{BH}  that 
for a   domain $D$,  the following are equivalent:

$(1)$ Each  proper ideal $A$ of $D$ has a comaximal factorization
$A = Q_1Q_2\cdots  Q_n$ where the $Q_i$'s have prime radical (resp. are primary, resp. are prime powers).

$(2)$ The prime spectrum of $D$ is a tree under inclusion and each  ideal of $D$ has only finitely many minimal primes (resp. $D$ is one dimensional and each  ideal of $D$ has only finitely many minimal primes, resp. $D$ is a Dedekind domain). Some related work was done in \cite{EGZ}.

The aim of this paper is to show that most of the results in \cite{BH} can be obtained in the setup of multiplicative lattices (also called abstract ideal theory). Our main results are Theorems \ref{5}, \ref{4},  \ref{11}, \ref{12}, \ref{15} and \ref{22}.

\section{Basic facts}

We recall some basic definitions and facts. For details, the reader may consult \cite{A} and \cite{D}.
The basic concept we use is:

\begin{definition} A {\em multiplicative lattice} is a  ordered multiplicative
commutative monoid $(L, \leq, \cdot)$ with the following properties:

$(1)$ $(L, \leq)$ is a complete lattice with   top element $1$ and  bottom
element $0.$ 

$(2)$ $x( \bigvee Y) = \bigvee_{y\in Y} xy$  for each $x\in L$ and $Y\subseteq L$ (here $\bigvee Y$ means $\bigvee_{y\in Y} y$).

$(3)$ $x1=x$ for all  $x\in L$.
\end{definition}
We always assume that  $0\neq 1$.
As usual,  the lattice operations $\vee$ and $\wedge$ are called {\em join } and {\em meet} respectively.
A subset $G$ of $L$ {\em generates} $L$ if every $x\in L$ is the join of some subset of $G$. An element $c\in L$ is {\em compact} if  
$c\leq \bigvee H$ with $H\subseteq L$ implies $c\leq \bigvee F$ for some finite subset $F$ of $H$. 
\\[1mm]
{\em In this paper, by {\em \underline{lattice}} we mean a multiplicative lattice $L$ such that}

$(i)$ $1$ is compact, 

$(ii)$ $L$ is generated by the set of compact elements,

$(iii)$ every product of two compact elements is compact.
\\

Let $L$ be a lattice.
An element    $x\in L$ is {\em proper} if $x\neq 1$.
Let $Max(L)$ denote the set of (proper) {\em maximal elements} of $L$. Every proper element is smaller than some maximal element.  A proper element $p$ is {\em prime} if  $xy \leq p$ with $x,y \in L$ implies $x \leq p$ or $y\leq p$.   Every maximal element is prime. $L$ is a {\em lattice domain} if $0$ is a prime element.
The {\em spectrum} $Spec(L)$ of $L$ is the set of all prime elements of $L$. The {\em dimension} of $L$ is the supremum of all integers $n\geq 0$ such that there exists a chain of prime elements $p_0<p_1<...<p_n$.
The {\em radical} of an element  $a\in L$ is
$$ \sqrt{a}=\bigvee \{ x\in L\ | \ x^n\leq a \mbox{ for some } n\geq 1   \}.$$ Then $\sqrt{a}$ is the meet of all primes $p\geq a$. An element $q\in L-\{1\}$ is {\em primary} if $xy\leq q$ with $x,y\in L$ implies $x\leq q$ or $y\leq \sqrt{q}$.
 For  $x,y\in L$, their quotient  is  $$(y : x) =\bigvee \{a \in L ;\ ax \leq y\}.$$
 An element $m$ is {\em meet-principal} (resp. {\em weak meet-principal})  
 if  $$a \wedge bm = ((a : m) \wedge b)m  \ \ \ \forall  a, b \in L \mbox{ (resp. } m\wedge a=(a:m)m\ \  \forall a\in L).$$
An element  $j$ is {\em join-principal} (resp. {\em weak join-principal})  if $$((aj \vee b) : j) = a \vee (b : j)  \ \ \forall  a, b \in L \mbox{ (resp. } (aj:j)=a\vee (0:j) \ \ \forall  a\in L).$$
 And $p\in L$ is {\em principal} if $p$ is both meet-principal and join-principal. The principal elements are compact. 
 If $x$ and $y$ are principal elements, then so is $xy$. The converse is also true if $L$ is a lattice domain and $xy\neq 0$. 
 In a lattice domain, every nonzero principal element is cancellative.

Say that two elements $a,b\in L$   are {\em comaximal} if $a\vee b=1$.
The following lemma (used hereafter without notice) collects some simple comaximality facts.

\begin{lemma} \label{1} 
Let $L$ be a lattice and $a,b,c_1,...,c_n\in L$.

$(i)$ If $a\vee b=1$, then $a\wedge b=ab$ and $(a:b)=a.$

$(ii)$ $a\vee b=1$ iff $\sqrt{a}\vee \sqrt{b}=1$ iff $a^i\vee b^j=1$ for some (all) $i,j\geq 1$.

$(iii)$ If $a\vee c_i=1$ for $i=1,...,n$, then $a\vee (c_1\cdots c_n)=1$
\end{lemma}
\begin{proof}
$(i)$ We have 
$a\wedge b = (a\wedge b)(a\vee b) \leq (a\wedge b)a\vee (a\wedge b)b \leq ab\leq a\wedge b
$ and $(a:b)=(a:(a\vee b))=(a:1)=a.$

$(ii)$ If $m$ is maximal element, then $a\vee b\leq m$ iff $\sqrt{a}\vee \sqrt{b}\leq m$ iff $a^i\vee b^j\leq m$.

$(iii)$ We have
$ 1  = (a\vee c_1)\cdots (a\vee c_n)\leq a\vee (c_1\cdots c_n).$
\end{proof}

For further use, we prove two useful formulas.
\begin{lemma}\label{9} 
Let $L$ be a lattice.

$(i)$ If $(a_{1k})_{k\geq 1}$,...,$(a_{nk})_{k\geq 1}$ are increasing sequences of elements of $L$, then
$$(\bigvee_{k\geq 1} a_{1k})\wedge \cdots \wedge (\bigvee_{k\geq 1} a_{nk})=
\bigvee_{k\geq 1} (a_{1k}\wedge \cdots \wedge  a_{nk}).
$$

$(ii)$ If $b,c\in L$ with $c$   compact, then 
$(\sqrt{b}:c) = \sqrt{\bigvee_{k\geq 1} (b:c^k)}$.
\end{lemma}
\begin{proof}
$(i)$ It suffices to prove that $"\leq"$ holds. Let $f$ be a compact $\leq $ LHS. Then $f\leq a_{1k}\wedge \cdots \wedge  a_{nk}$ for some $k\geq 1$, because $(a_{1k})_{k\geq 1}$,...,$(a_{nk})_{k\geq 1}$ are increasing sequences, so $f\leq $ RHS.

$(ii)$ Let $f$ be a compact $\leq $ LHS. We get succesively: $(fc)^k\leq b$ for some $k\geq 1$,  $f^k\leq (b:c^k)$,  $f\leq \sqrt{(b:c^k)}$,  $f\leq $ RHS. Conversely, let  $f$ be a compact $\leq $ RHS. Then $f^j\leq (b:c^k)$ for some $j,k\geq 1$ where we can arrange $j=k$. 
We get $(fc)^k\leq  b$, so  $f\leq (\sqrt{b}:c)$, thus $f\leq $ LHS.
\end{proof}

\section{Results}

The aim of this paper is to study the comaximal factorizations with additional properties in lattices, thus extending the work done in \cite{BH}.  We start by establishing a key result about the comaximal factorizations with prescribed radicals.

\begin{theorem} \label{5} 
Let  $L$ be a lattice, $a_1,...,a_n\in L-\{1\} $   pairwise comaximal elements and set $a=a_1a_2  \cdots  a_n$. Then every $b\in L$ with 
$\sqrt{b}=\sqrt{a}$ can be uniquely written as a product of some  pairwise comaximal elements $b_1,...,b_n$ such that $\sqrt{b_i}=\sqrt{a_i}$ for $i=1,...,n$. In particular, if $a$ is a radical element, then $a_1$,...,$a_n$ are radical elements.
\end{theorem}
\begin{proof} 
Since $a_1,...,a_n$ are  pairwise comaximal, we have 
$$\sqrt{a}=\sqrt{a_1\cdots a_n}=
\sqrt{a_1}\wedge \cdots \wedge \sqrt{a_n}  
=\sqrt{a_1}\cdots  \sqrt{a_n}$$ 
so we may assume that $a=\sqrt{a}$ and $\sqrt{a_i}=a_i$ for each $i$.
For $i$ between $1$ and $n$, produce the elements $c_i$, $b_i$ as follows. Since  $a_i$ is comaximal to 
$\prod_{j\neq i}a_j $, it follows that $a_i\vee c_i=1$ for some compact element $c_i\leq \prod_{j\neq i}a_j $. It follows that
\begin{equation}\label{10} 
(a:c_i) = ((a_1\wedge  \cdots \wedge a_n):c_i)=
(a_1:c_i)\wedge  \cdots \wedge (a_n:c_i) = (a_i:c_i) = a_i.
\end{equation}
Set $b_i=
\bigvee_{k\geq 1} (b:c_i^k)$. 
By Lemma \ref{9} and equality $(\ref{10})$, we get
\begin{equation}\label{3a} 
\sqrt{b_i}= \sqrt{\bigvee_{k\geq 1} (b:c_i^k)}=(\sqrt{b}:c_i)=(a:c_i)=a_i.
\end{equation}
Since $a_i\vee c_i=1$ for each $i$, it follows that $a=a_1 \cdots  a_n$ is comaximal to $c_1^k\vee \cdots \vee c_n^k$, so
\begin{equation}\label{3} b\vee c_1^k\vee \cdots \vee c_n^k=1 \end{equation}
because $\sqrt{b}=a$.
By equality $(\ref{3a})$, Lemma \ref{9} and equality $(\ref{3})$, we get
$$ b_1   \cdots  b_n= b_1\wedge  \cdots \wedge b_n=
\bigvee_{k\geq 1} (b:c_1^k)\wedge  \cdots \wedge 
\bigvee_{k\geq 1} (b:c_n^k)=
$$
$$ 
=\bigvee_{k\geq 1} ((b:c_1^k)\wedge  \cdots \wedge 
 (b:c_n^k))=\bigvee_{k\geq 1} (b:(c_1^k\vee \cdots \vee c_n^k))=b.
$$
We prove the uniqueness. Suppose that $b=d_1\cdots d_n$ for some pairwise comaximal elements $d_1,...,d_n$ such that $\sqrt{d_i}=a_i$ for $i=1,...,n$. Note that $(d_i:c_i^k)=d_i$ because $a_i\vee c_i=1$. Moreover, for $i\neq j$, we have that $c_i\leq a_j$, so $c_i^k\leq d_j$ for some $k\geq 1$. 
Using these remarks and  Lemma \ref{9}, we get
$$b_i=\bigvee_{k\geq 1} (b:c_i^k) =  \bigvee_{k\geq 1} ((d_1\wedge \cdots \wedge d_n):c_i^k) = \bigvee_{k\geq 1} ((d_1:c_i^k)\wedge  \cdots \wedge 
 (d_n:c_i^k))= 
$$
$$ 
=\bigvee_{k\geq 1} (d_1:c_i^k)\wedge  \cdots \wedge 
\bigvee_{k\geq 1} (d_n:c_i^k)= \bigvee_{k\geq 1} (d_i:c_i^k)=d_i.
$$
The "in particular" statement follows from this uniqueness property.
\end{proof}

We give the   key definition of this paper.

\begin{definition}\label{14}
Let $L$ be a lattice. 
For an element  $a\in L-\{1\}$, consider a factorization 
$$
(\sharp)\ \ \ \ \ \   a=a_1a_2\cdots a_n \mbox{ with } a_1,...,a_n 
 \mbox{ pairwise comaximal elements from } L-\{1\}.
$$
We say that 
$(\sharp)$ is a:

$(i)$  {\em comaximal prime radical factorization  (CPR-factorization)} of $a$ if $\sqrt{a_1}$,...,$\sqrt{a_n}$ are prime elements,

$(ii)$  {\em comaximal primary factorization  (CQ-factorization)} of $a$ if ${a_1}$,...,${a_n}$ are primary elements,

$(iii)$  {\em comaximal prime power factorization  (CPP-factorization)} of $a$ if ${a_1}$,...,${a_n}$ are prime powers.

Next say that $L$ is a {\em CPR-lattice}, {\em CQ-lattice}, {\em CPP-lattice}, if every $x\in L-\{1\}$ has a CPR-factorization, CQ-factorization, CPP-factorization respectively.
\end{definition}

We clearly have
$$ (*)\ \ \ \  \ \ \ \ \ \ \mbox{CQ-lattice} \ \ \  \Rightarrow \ \ \  \mbox{CPR-lattice}\ \ \ \Leftarrow\ \ \  \mbox{CPP-lattice}$$

 These three conditions were studied in \cite{BH} for the ideal lattice of an integral  domain. To see that  in $(*)$ no other implication holds, let us examine the following four finite lattices $L_1$, $L_2$, $L_3$, $L_4$ having underlying set $\{0,1,a,b,c,d \}$ ordered by 
$a\leq b\leq d$ and  $a\leq c\leq d$.

\begin{example}
Let   $L_1$ be the lattice  with multiplication 
$$ a^2=ab=ac=ad=bc=c^2=cd=a,\ \  b^2=d^2=bd=b.$$ Its primes are $0,c$ and $d$, hence $a=c^2$, $b=d^2$ are CPP-factorizations. So $L_1$ is a CPP-lattice but  not a CQ-lattice, because $a$ is not primary as $cd=a$, $c\not\leq a$ and $d\not\leq c=\sqrt{a}$. 
\end{example}

\begin{example}
Let  $L_2$  be the lattice  with multiplication  
$$xy=a \mbox{ for all }  x,y\in \{a,b,c,d\}.$$
Its primes are $0$ and $d$, hence $a,b,c$ are primary elements as their radical is the maximal element $d$. So $L_2$ is a CQ-lattice but is not a CPP-lattice because $b$ and $c$ are not prime powers. 
\end{example}

\begin{example}
Let  $L_3$  be the lattice  with multiplication  
$$a^2=ab=ac=ad=bc=a,\ \ b^2=bd=b,\ \ c^2=cd=c, \ \ d^2=d.  $$
Its primes are $0,b,c$ and $d$. As $a=\sqrt{a}$ has no CPR-factorization, $L_3$ is not a CPR-lattice. 
\end{example}

\begin{example}
Let  $L_4$  be the lattice  with multiplication  
$$a^2=ab=b^2=0,\ \ ac=ad=bc=bd=a,\ \ c^2=cd=d^2=c.  $$
Its primes are $b$ and $d$, so the radical elements are prime, hence $L_4$ is a CPR-lattice. Note that $a$ is neither a prime power nor primary since $bc=a$, $b\not\leq a$, $c\not\leq b=\sqrt{a}$.
Hence $L_4$ is neither a CPP-lattice nor a CQ-lattice. 
\end{example}

Say that a lattice $L$ is a {\em treed lattice} if its spectrum  is a tree as an ordered set under inclusion (equivalently, if any two incomparable prime elements of $L$ are comaximal).  
Note that in a treed lattice, the elements of $Min(a)$ are pairwise comaximal for each $a\in L-\{1\}$. Here $Min(a)$ is the set of minimal primes over $a$.
The following result extends \cite[Theorem 2]{BH}.

\begin{theorem} \label{4} 
Let $L$ be a lattice. 

$(i)$ An element   $a\in L-\{1\}$  has 
a CPR-factorization iff 
$Min(a)$ is finite and its elements are pairwise comaximal; if it exists, the CPR-factorization of $a$ is unique.

$(ii)$  $L$ is a CPR-lattice iff $L$ is   treed  and $Min(a)$ is finite for all $a\in L-\{1\}$.
\end{theorem}
\begin{proof} 
$(i)$ We use repeatedly Theorem \ref{5}; we may suppose that $a=\sqrt{a}$.
If $a=p_1\wedge \cdots \wedge p_n$ is a CPR-factorization of $a$, then 
 $p_1$,...,$p_n$ are pairwise comaximal primes, hence $Min(a)=\{p_1,...,p_n\}$ and the CPR-factorization of $a$ is unique. Conversely, if 
$Min(a)=\{q_1,...,q_m\}$ and $q_1,...,q_m$ are pairwise comaximal, the CPR-factorization of $a$ is $q_1\wedge \cdots \wedge q_m$.

$(ii)$ follows from $(i)$ and the fact that $Min(p\wedge q)=\{p,q\}$ whenever $p,q$ are two incomparable primes.
\end{proof}

\begin{corollary} \label{20}
If $L$ is a treed lattice, then  the set of elements in $L$ having a CPR-factorization is closed under products, finite meets and finite joins.
\end{corollary}
\begin{proof} 
Let $\Gamma$ be  the set of elements in $L$ having a CPR-factorization.
By the proof of Theorem \ref{4},  $x\in \Gamma$  iff $Min(x)$ is finite. Let $x,y\in \Gamma$. We prove that $Min(x\vee y)\subseteq  Min(x)\cup Min(y)$, so $Min(x\vee y)$ is finite. Let $p\in Min(x\vee y)$. Then $p_1\vee p_2\leq p$ for some $p_1\in Min(x)$ and $p_2\in Min(y)$. As $L$ is treed, we may assume that $p_1\leq p_2$. Then $x\vee y\leq p_2\leq p$, so $p=p_2$ since $p$ is minimal over $x\vee y$. 

The fact that $Min(xy)=Min(x\wedge y)\subseteq  Min(x)\cup Min(y)$ follows easily from definitions. So $xy,\ x\wedge y\in \Gamma$.
\end{proof}

Let $D$ be an integral domain and $M$ its ideal lattice. By \cite[Theorem 1]{BH}, $M$ is treed provided every proper principal ideal of $D$ has a CPR-factorization. We  extend that result.

\begin{theorem}\label{11} 
Let $L$ be a lattice and $G\subseteq L$ a generating set for $L$. 
If  $g_1g_2$ has a CPR-factorization for every $g_1,g_2\in G-\{1\}$, then $L$ is a  treed lattice.
\end{theorem}
\begin{proof} 
Suppose that $L$ is not  treed, so  $p_1\vee p_2\neq 1$ for two incomparable primes   $p_1$ and $p_2$. Since $G$ generates $L$, there exist $g_1,g_2\in G$ such that 
$$g_1\leq p,\ \  g_1\not\leq p_2,\ \  g_2\leq p_2,\ \  g_2\not\leq p_1$$
and let $g_1g_2=a_1\cdots a_n$ be a CPR-factorization of $g_1g_2$. Since $g_1g_2\leq p_1p_2$ and $p_1,p_2$ 
are prime, we get that $a_i\leq p_1$ and $a_j\leq p_2$ for some $i,j$. As $p_1\vee p_2\neq 1$ and $a_1,...,a_n$ are pairwise comaximal, we get $i=j$. As $g_1g_2\leq a_i$ and $\sqrt{a_i}$ is prime, we may assume that $g_1\leq \sqrt{a_i}$. 
Then $g_1\leq \sqrt{a_i} \leq p_2$ which is a contradiction.
\end{proof}


\begin{corollary} \label{6}  
Let $L$ be a lattice and $G\subseteq L$ a generating set for $L$ consisting of compact elements.
The following  are equivalent.

$(i)$  Every compact element   $k\in L-\{1\}$ has a CPR-factorization

$(ii)$ $g_1g_2$ has a CPR-factorization for every $g_1,g_2\in G-\{1\}$

$(iii)$ $L$ is treed and $Min(k)$ is finite for every compact element $k\in L$.

$(iv)$ $L$ is treed and $Min(g)$ is finite for every $g\in G$.
\end{corollary}
 \begin{proof}
The implications $(i)\Rightarrow (ii)$  and $(iii)\Rightarrow (iv)$ are obvious. $(ii)\Rightarrow (iv)$ follows from Theorem \ref{11} and the fact that $Min(g^2)=Min(g)$ is finite when $g^2$ has a CPR-factorization.
$(iv)\Rightarrow (ii)$ follows from Corollary  \ref{20}. Finally 
$((i)$ and $(iv))\Rightarrow (ii)$ by Theorem \ref{4}.
 \end{proof}
Using the translation done in \cite[Section 3]{AC}, we can see that \cite[Theorem 2.1]{EGZ} is a consequence of the above corollary.

Recall that a {\em Bezout domain} is a domain whose finitely generated ideals are principal. In \cite[Example 7]{BH}, the authors constructed a Bezout domain $D$ having the following two properties: 

$(1)$ Each finitely generated  ideal of $D$ has only finitely many minimal primes,

$(2)$ $D$ has an  ideal having infinitely many minimal primes. 
\\
It is well-known that the spectrum of a Bezout domain is a tree under inclusion. So the ideal lattice $K$ of $D$ satisfies 
 condition $(iii)$ in Corollary \ref{6}, but $K$ is a not a CPR-lattice.
Getting inspiration from \cite[Proposition 4]{BH}, we exhibit a case when 
we do obtain a CPR-lattice provided the hypothesis of  Theorem \ref{11} holds.

\begin{theorem} \label{12} 
Let $L$ be a lattice and $G\subseteq L$  a generating set for $L$. 
Suppose that the following three conditions  hold.

$(1)$ Every non-minimal prime element  of $L$  is below only finitely many maximal elements. 

$(2)$ If $a\in L$ and $p_1,...,p_n\in Spec(L)$ such that 
$a\not\leq p_i$ for $i=1,...,n$, there exists some $g\in G$ such that 
$g\leq a$ and $g\not\leq p_i$ for $i=1,...,n$.

$(3)$  $gh$ has a CPR-factorization for every $g,h\in G-\{1\}$.
\\
Then $L$ is a CPR-lattice.
\end{theorem}
\begin{proof} 
By Theorems \ref{11} and \ref{4}, it follows that $L$ is a treed lattice 
and $Min(g^2)=Min(g)$ is finite for all $g\in G$.
By Theorem \ref{4}, it suffices to show that  $Min(x)$ is finite for all 
$x\in L-\{1\}$.
Suppose on the  contrary that $Min(a)$ is infinite for some $a\in L-\{1\}$. Then $Min(0)\not\subseteq Min(a)$, otherwise we get $Min(a)=Min(0)$ which is a contradiction since $0\in G$ so $Min(0)$ is finite. 
 
By $(2)$, there exists   some $h\in G$ such that $h\leq a$ and $h\not\leq q$ for all $q\in Min(0)-Min(a)$. 
For each $p\in Min(a)-Min(0)$, pick  a maximal element $p'\geq p$. Note that $p'=q'$ implies $p=q$ because $L$ is treed. 
We get the infinite set
  $$\Gamma=\{p'\ |\ p\in Min(a)-Min(0) \}.$$ 
Let $p'\in \Gamma$. Then $h\leq a\leq p'$, so $\dot p\leq p'$ for a unique element $\dot p\in Min(h)$, because $L$ is treed. Note that $\dot p\notin Min(0).$
Indeed, if $\dot p\in Min(0)$, then $\dot p\in Min(a)$ by our choice of $h$, so $p$ and $\dot p$ are distinct elements of $Min(a)$, thus getting the contradiction $1=p\vee \dot p\leq p'$. 

Consider  the map 
 $f:\Gamma\rightarrow Min(h)-Min(0)$  given by $f(p')=\dot p$. As $\Gamma$ is infinite while $Min(h)$ is finite, there exits some $q\in Min(h)-Min(0)$ with $f^{-1}(q)$ infinite. Hence $q$ is a non-minimal prime element which is   below infinitely many maximal elements, thus contradicting hypothesis $(1)$.
\end{proof}

\begin{corollary}  
Let $L$ be a lattice such that every non-minimal prime element  of $L$  is below only finitely many maximal elements. 
If every compact element $c\in L-\{1\}$ has a CPR-factorization, then $L$ is a CPR-lattice.
\end{corollary}

\begin{remark}\label{13}  
Note that hypothesis $(2)$ in Theorem \ref{12} holds automatically if $L$ has only one minimal prime element, e.g. when $L$ is a lattice domain.
\end{remark}

We switch to CQ-lattices (see Definition \ref{14}).

\begin{theorem} \label{17}
For a lattice $L$, the following are equivalent:

$(i)$ $L$ is a CQ-lattice.

$(ii)$ $L$ is a CPR-lattice and every element of $L$ with prime radical is primary.
\end{theorem}
\begin{proof} 
$(ii)\Rightarrow (i)$ is clear from definitions. $(i)\Rightarrow (ii)$. If $a\in L$ has prime radical, then $a$ is primary because its unique CQ-factorization (see Theorem \ref{4}) is $a=a$.
\end{proof}

\begin{corollary} \label{18}
For a lattice domain $L\neq \{0,1\}$ generated by compact join principal elements, the following are equivalent:

$(i)$ $L$ is a CQ-lattice.

$(ii)$ $L$ is one-dimensional  and $Min(a)$ is finite for each $a\in L$.
\end{corollary}
\begin{proof} 
$(i)\Rightarrow (ii)$. By \cite[Corollary 3.4]{AAJ}, a lattice domain $L$ is one-dimensional provided $L$ is generated by compact join principal elements and every element of $L$ with prime radical is primary.
So the assertion follows combining Theorems \ref{17}, \ref{4} and \cite[Corollary 3.4]{AAJ}.

$(ii)\Rightarrow (i)$ is covered by the following simple lemma.
\end{proof}

\begin{lemma}\label{24}
If $L$ is a one-dimensional lattice domain and $Min(a)$ is finite for each $a\in L$, then $L$ is a CQ-lattice.
\end{lemma}
\begin{proof} 
Let $a\in L-\{1\}$. As $0$ is prime, we may assume that $a\neq 0$. Then $\sqrt{a}=m_1\wedge \cdots \wedge m_n$ where 
$m_1,...,m_n$ are the maximal elements greater than $a$. By Theorem \ref{5}, $a$ has a CPR-factorization $a=a_1\cdots a_n$ where $\sqrt{a_i}=m_i$, $i=1,...,n$. Since an element with maximal radical is primary, we are done.
\end{proof}

Using ideas from the proof of \cite[Theorem 8]{BH}, we can 
slightly sharpen Corollary \ref{18} as follows.

\begin{theorem}\label{15}  
Let $L\neq \{0,1\}$ be a lattice domain. Suppose that $L$ is generated by a set  $G$  such that $(ab:a)\leq \sqrt{b}$ for every $a,b\in G-\{0\}$. The following   are equivalent.

$(i)$ $ab$ has a CQ-factorization for every $a,b\in G-\{1\}$.

$(ii)$ $L$ is one-dimensional  and $Min(a)$ is finite for each $a\in L$.

$(iii)$ $L$ is a CQ-lattice.
\end{theorem}
\begin{proof}
$(iii)\Rightarrow (ii)$ 
We show that $L$ is one-dimensional. Suppose on the  contrary that there exist prime elements $0<p<m$. We can find $a,b\in G$ such that 
$$0<a\leq p,\ \  b\leq m \ \mbox{ and }\   b\not\leq p.$$ 
Let $ab=q_1\cdots q_n$ be a CQ-factorization of $ab$. 
As $ab\leq a \leq p$, we get  that  $p\geq q_i$ for some $i$,  say $p\geq q_1$. Since $b\not\leq p$ and $\sqrt{q_1}\leq p$, we get that $b \not\leq \sqrt{q_1}$. 
As $ab\leq q_1$ and $q_1$ is primary, it follows that $a\leq q_1$. 
Since $ab=q_1\cdots q_n\leq b\leq m$ and $q_1,...,q_n$ are pairwise comaximal, we get that $q_2\cdots q_n\not\leq  m$. From  
$$aq_2\cdots q_n \leq q_1\cdots q_n =ab,  
$$ and by our hypothesis about $G$, we get $$q_2\cdots q_n \leq 
(ab:a)\leq \sqrt{b}\leq m,  
$$
 which is a contradiction. So $L$ is one-dimensional. To complete,  apply Theorem \ref{12}, Remark \ref{13} and Theorem \ref{4}.
 
 $(ii) \Rightarrow (iii)$ is Lemma \ref{24} and 
$(iii) \Rightarrow (i)$ is obvious. 
\end{proof}

The next example is a finite CQ-lattice domain of dimension $2$.

\begin{example}\label{16}
Let $L$ be the lattice with underlying set $\{0,1,a,b,c,d \}$ ordered by 
$a\leq b\leq c$,  $b\leq d$ whose multiplication is the meet.
Its primes are $0,a,c,d$ so $L$ is a  CQ-lattice, because $b=cd$ is a CQ-factorization of $b$.
\end{example}

We switch to CPP-lattices (see Definition \ref{14}).
Let $L$ be a lattice domain generated by principal elements. 
Recall 
that  $L$ is a {\em Dedekind lattice}  if every element of $L$ is a finite product of prime elements (see for instance \cite[Definition 2.1]{AJ}). A Dedekind lattice is one-dimensional \cite[Theorem 2.6]{AJ}. Moreover, $L$ is a Dedekind lattice iff every element is principal (see for instance the remark before Theorem 3.1 in \cite{AJ}).
 Though the following result is well-known, we insert a proof for reader's convenience.

\begin{lemma}\label{8}  
Let $L$ be  lattice domain generated by principal elements. If every prime element of $L$ is principal, then $L$ is a Dedekind lattice. 
\end{lemma}
\begin{proof} 
Suppose by contrary that there exists a non-principal element $x\in L$. Using a Zorn's Lemma argument, we may assume that $x$ is maximal with this property. Then $x$ is not prime, so there exist $a,b\in L$ such that $x<a$, $x<b$ and $ab\leq x$. 
By the maximal choice of $x$, we get that $a$ and $(a:x)$ are principal elements because
$x< b \leq (x:a)$. Since $a$ is principal and $x < a$, we get that $x=a(x:a)$ is principal, a contradiction.
\end{proof}

The next result extends \cite[Theorem 9]{BH}.

\begin{theorem}\label{22} 
Let $L$ be  lattice domain generated by principal elements. Then $L$ is a Dedekind lattice iff every nonzero principal element $x$ has a CPP-factorization.
\end{theorem}
\begin{proof} 
 $(\Rightarrow)$ follows from the paragraph before Lemma \ref{8}.

$(\Leftarrow)$ We claim that $L$ is one-dimensional. Suppose by contrary that
there exist two nonzero prime elements $p<m$. Let $x\leq m$ be a principal element such that $x\not\leq p$. Since  $x$ has a   CPP-factorization, $x$ is a product of primes $x=p_1\cdots p_n$. Note that each factor $p_i$ is principal.  Since $p_1\cdots p_n \leq m$, it follows that $m\geq p_i$ for some $i$. Replacing $x$ by that $p_i$, we may assume that $x$ is prime. 
Repeating this argument, we can produce a nonzero principal prime $y\leq p$.
By Corollary \ref{6}, it follows that $y$ and $x$ are comparable, so $y<x$. Then $y=x(y:x)=xy$,  because $y$ is prime. We get $1=(xy:y)=x$, a contradiction. Thus $L$ is one-dimensional. 
Let $r$ be a nonzero prime element of $L$. As argued in the beginning of our proof, there exists a  nonzero principal prime $z\leq r$. By one-dimensionality, we get that $r=z$ is principal, so Lemma \ref{8} applies.
\end{proof}

Note that the lattice in Example \ref{16} is a non-Dedekind CPP-lattice.


\end{document}